\documentclass[11pt]{article}
\pdfoutput=1 
\usepackage{a4}
\usepackage{amsmath}
\usepackage{amssymb}
\usepackage{amsthm}
\usepackage{lipsum}
\usepackage{amsfonts}
\usepackage{graphicx}
\usepackage{epstopdf}
\usepackage{epsfig}
\usepackage{pgfgantt}
\usepackage[numbers, sort&compress]{natbib}
\usepackage{cite}
\usepackage{enumerate}
\usepackage{parskip}
\usepackage{cleveref}
\usepackage[affil-it]{authblk}

\newtheorem{thm}{Theorem}
\newtheorem{lemma}{Lemma}

\newtheorem{dfn}{Definition}

\newtheorem{prop}{Proposition}

\title{Strictly positive definite kernels on the $2$-sphere: beyond radial symmetry}

\author{ Janin J\"ager \\ \tiny{
janin.jaeger@math.uni-giessen.de} }
 \affil{
Lehrstuhl Numerische Mathematik,
Justus-Liebig University, 35392 Giessen, Germany.}

\date{\today}

\begin{document}
\def\R{{\mathbb R}}
\def\Z{{\mathbb Z}}
\def\N{{\mathbb N}}
\def\S{{\mathbb{S}^2}}
\def\wh{\widehat}
\def\C{{\mathbb C}}

\maketitle

\begin{abstract}
The paper introduces a new characterisation of strict positive definiteness for kernels on the 2-sphere without assuming the kernel to be radially (isotropic) or axially symmetric. The results use the series expansion of the kernel in spherical harmonics. Then additional sufficient conditions are proven for kernels with a block structure of expansion coefficients. These generalise the result derived by Chen et al.\ \citep{Chen2003} for radial kernels to non-radial kernels.
\end{abstract}

\textit{keywords:} positive definite kernels, covariance functions, two-sphere, \\
\textit{2010 MSC:} 33B10; 33C45; 42A16; 42A82; 42C10

\section{Introduction}
In the last five years, there has been a tremendous number of publications stating new results on positive definite kernels on spheres, see for example  \citep{Beatson2014, Gneiting2013,Hubbert2015,Xu2018, Truebner2017, Arafat2018, Bissiri2019,Nie2018, Berg2017}. Most of these focus on isotropic positive definite kernels, which are kernels that only depend on the geodetic distance of their arguments.  Isotropic kernels  are used in approximation theory, where they are often referred to as spherical radial basis functions \citep{Castell2004, Hubbert2001,Beatson2018,Beatson2017}  and are for example applied in geostatistics and physiology \citep{Jaeger2016, Fornberg2015}. They are also of importance in statistics where they occur as correlation functions of homogeneous random fields on spheres \citep{Lang2015}.

Recently  the use of axial-symmetric kernels on the sphere was studied and used for the approximation of global weather data in  \citep{Bissiri2019}. This publication will characterise even more general positive definite kernels on the sphere. We give necessary and sufficient conditions for kernels which are only required to be continuous and Hermitian.  The general results will allow us to reproduce results achieved  earlier easily for isotropic and axially symmetric kernels. Further, we will in Section 3 characterise a new class of positive definite spherical kernels which are not isotropic but possess certain desirable smoothness properties.

We will briefly summarize necessary definitions in the first section and then state the general results for characterisation of positive definite spherical kernels in Section 2. In the third section, we describe the connection to the existing results and add a new class of strictly positive definite kernels.

\subsection{Problem description 
and Background}

We focus on interpolation problems on the $2$-sphere 
$$\mathbb{S}^2=\lbrace \xi \in \R^3 \vert \xi_1^2+\xi_2^2+\xi_3^2=1 \rbrace,$$
 where a finite set of distinct data sites $\Xi\subset \mathbb{S}^2$ and given values $f(\xi) \in \C$, $\xi \in \Xi$, of a possibly elsewhere unknown function $f$ on the sphere are given.

The approximant is formed as a linear combination of kernels \mbox{$K:\mathbb{S}^2 \times \mathbb{S}^2 \rightarrow \C$.} Taking the form 
\begin{equation}\label{eq:Interpolant}  s_f(\zeta)=\sum_{\xi\in \Xi} c_{\xi} K(\zeta, \xi), \qquad  \zeta \in \mathbb{S}^2.
\end{equation}
The  problem of finding such an approximant $s_f$ satisfying
 \begin{equation}
 s_f(\xi)=f(\xi), \qquad \forall \xi \in \Xi,
 \end{equation}
 is uniquely solvable under certain conditions on $K$, which we  will now introduce. We assume all the kernels to be Hermitian, meaning they satisfy $K(\xi,\zeta)=\overline{K(\zeta, \xi)}$ so that the positive definiteness of the kernel will ensure the solvability of the interpolation problem for arbitrary data sets.
 \begin{dfn}
Let $H$ be a $\C$-Hilbert space with inner product $\langle \cdot , \cdot \rangle$. Then a linear and continuous operator $P:H\rightarrow H$ is called positive if it is self adjoint and 
$$\langle P(x),x \rangle \geq0 ,\qquad  \forall x \in H.$$
A linear and continuous operator $P:H\rightarrow H$ is called strictly positive if it is self adjoint, bounded  and 
$$\langle P(x),x\rangle >0 ,\qquad  \forall x \in H\setminus \lbrace 0\rbrace .$$
 \end{dfn}
\begin{dfn}\label{DF:SPD}
A Hermitian kernel $K:\mathbb{S}^2\times \mathbb{S}^{2}\rightarrow \C$ is {\/\rm
positive definite on} $\mathbb{S}^{2}$ if and only if the matrix  $K_{\Xi}=\left\lbrace
K(\xi,\zeta))\right\rbrace_{\xi,\zeta \in \Xi}$ induces a positive operator on $\C^{\vert \Xi\vert}$ via left  multiplication for arbitrary finite sets of distinct points $\Xi\subset \S$.

The kernel is {\/\rm strictly positive definite} if $K_{\Xi}$ induces a strictly positive operator on $\C^{\vert \Xi\vert}$ via left  multiplication  for  arbitrary finite sets of distinct points $\Xi$.
\end{dfn} 
The last definition deviates from the standard notation used in approximation theory because we define positive definiteness via the properties of an operator instead of the matrix which defines the operator. The gain in our specific setting is that we can easily switch between finite and infinite-dimensional Hilbert spaces using the operator notation and that we do not have the usual notational inconsistency of defining positive definite kernels using semi-definite matrices. 
We assume throughout this paper that $K$ is continuous in both arguments, it is thereby square-integrable and can be represented as
 \begin{equation}\label{eq:Kernform} K(\xi, \zeta)=\sum_{\ell,\ell'=1}^{\infty}a_{\ell,\ell'}F_{\ell}(\xi)\overline{F_{\ell'}(\zeta)},\qquad \forall \xi,\zeta \in \S ,
 \end{equation}
where the $F_{\ell}$ form an orthonormal basis of the eigenfunctions of the Laplace-Beltrami operator on the sphere. 
The Laplace-Beltrami operator on the sphere is defined using polar coordinates as 
$$\Delta _{\S} f(\theta,\phi) = (\sin\phi)^{-1} \frac{\partial}{\partial \phi}\left(\sin\phi\frac{\partial f}{\partial \phi}\right) + (\sin\phi)^{-2} \frac{\partial^2}{\partial \theta^2}f.$$
Its eigenfunctions corresponding to the eigenvalues $k(k+1)$ are the spherical harmonics, we will in the third section change to the more common double index notation, $Y_k^m$,  of these functions. 
It is well known and used in the characterisation above, that every function in $L^2(\S)$ can be represented as a spherical harmonic expansion of the form 
\begin{equation}\label{eq:Serrep}
g(\xi)= \sum_{\ell=1}^{\infty}\hat{g}_{\ell}F_{\ell}(\xi),\qquad \text{ with } \hat{g}_{\ell}=\int_{\S}g(\xi)F_{\ell}(\xi)\, d\xi.
\end{equation}
For this expansion the Parseval equation for spherical harmonics  holds: 
$$\Vert g\Vert_{L^2(\S)}^2=\sum_{\ell=1}^{\infty} \vert \hat{g}_{\ell}\vert ^2.$$

We also define the vector space  
$$Y:=\operatorname{span}\left \lbrace \left( F_{\ell}(\xi)\right )_{\ell=1}^{\infty}  \big\vert  \xi \in \S \right \rbrace \subset \C^{\infty}.$$

Since we assume $K$ to be continuous the series \eqref{eq:Kernform} will converge for all $\xi,\zeta  \in \S $. Another consequence is that  
$$\int_{\S} \int_{\S}(K(\xi,\zeta))^2\,d\xi \, d\zeta < \infty, $$
which ensures $\sum_{\ell, \ell'=1}^{\infty} \vert a_{\ell,\ell'}\vert^2 < \infty$ because of orthonormality and thereby makes the linear operator 
\begin{equation}\label{eq:DefOp}
 A(y):=\left( \sum_{\ell=1}^{\infty}  a_{\ell,\ell'}y_{\ell}\right)_{\ell'=1}^{\infty},\qquad y\in\ell^2 \text{ or } y\in Y,
\end{equation}
bounded on $\ell^2 $, by which we denote the set of square-summable complex-valued sequences.
 We define the inner product 
 $$\langle x, y \rangle:=\sum_{\ell=1}^{\infty}\overline{x_{\ell}}y_{\ell}$$
 and will use it for $x,y\in \ell^2$ as well as for $x,y\in Y$. 
Even though we do not know whether  $A(y)$ is in $\ell^2$ for elements $y\in Y$, we know that any $y\in Y$ can be represented as $y= \left( \sum_{\xi\in \Xi}c_{\xi} F_{\ell}(\xi)\right )_{\ell=1}^{\infty}$  and therefore 
$$\langle A(y),y \rangle=\sum_{\xi,\zeta\in\Xi}c_{\xi}\overline{c_{\zeta}}K(\xi,\zeta)$$ 
is finite for all $y\in Y$.

Also we will need the definition of the Sobolev spaces $W_2^r(\S)$ to be the subspace of $L^2(\S)$ with finite norm
$$ \Vert f \Vert_{W^r_2(\S)}:= \left( \sum_{\ell=1}^{\infty} (1+\lambda_{\ell})^{r} \vert \hat{f}_{\ell}\vert^2 \right)^{1/2},$$
where $\lambda_{\ell}$ is the eigenvalue corresponding to the eigenfunction $F_{\ell}$. Further we need $W_2^r(\S \times \S)$ with norm  
$$ \Vert K \Vert_{W^r_2(\S\times \S)}:= \left( \sum_{\ell,\ell'=1}^{\infty} (1+\lambda_{\ell}+\lambda_{\ell'})^{r} \vert a_{\ell,\ell'}\vert^2\right)^{1/2}$$
as well as the continuous embedding lemma which says that the above Sobolev spaces are consisting only of continuous functions if $r>1$. 
\section{Characterisation of (strictly) positive definite kernels}
The first two results we state are closely connected to the results of Dyn, Narcowich and Ward \citep{Dyn1997} which were stated for more general manifolds but used distributions in certain Sobolev spaces instead of sequences to describe positive definiteness.  Our approach focuses more on the matrix-like structure of the operator $A$ which allows us in the third section to employ the techniques from linear algebra. Our first observation establishes the connection between $A$ and $K$.
\begin{lemma} Let $K: \S \times \S \rightarrow \C$ be a continuous kernel with $K(\xi,\zeta)=\overline{K(\zeta,\xi)}$. Then $K$ is positive definite if and only if $\langle A(y),y\rangle\geq 0$ for all $y\in Y$, where $A$ is as in \eqref{eq:DefOp}. 
\end{lemma}
\begin{proof} Let   $\langle A(y),y\rangle\geq 0$ for all  $y \in Y$, then for arbitrary sets of distinct points $\Xi \subset \S$ and $c\in \C^{\vert \Xi \vert}$ 
\begin{align*}
\sum_{\xi \in \Xi} \sum_{\zeta \in \Xi}c_{\xi} \overline{c_{\zeta}} K(\xi, \zeta)&= \sum_{\xi \in \Xi} \sum_{\zeta\in \Xi}c_{\xi} \overline{c_{\zeta} }\sum_{\ell,\ell'=1}^{\infty}a_{\ell,\ell'}F_{\ell}(\xi)\overline{F_{\ell'}(\zeta)} \\
& = \sum_{\ell,\ell'=1}^{\infty}a_{\ell,\ell'} \sum_{\xi \in \Xi}c_{\xi}F_{\ell}(\xi)\sum_{\zeta\in \Xi} \overline{c_{\zeta}F_{\ell'}(\zeta)}  \\
&=\sum_{\ell,\ell'=1}^{\infty}y_{\ell} a_{\ell, \ell'}\overline{y_{\ell'}}=\langle A(y),y\rangle \geq 0, \\
&\text{with }   y =\left(\sum_{\xi \in \Xi}c_{\xi}F_{1}(\xi), \sum_{\xi \in \Xi}c_{\xi}F_{2}(\xi),\ldots \right) \in Y.
\end{align*}
The order of summation in the second line can be exchanged because the series converge for all $\xi$ since $K$ is continuous.

To prove the opposite direction we assume there exists $y \in Y$ with $\langle A(y),y \rangle <0.$
Then there exists a finite set of points $\Xi \subset \S$ such that 
$$y=\left(\sum_{\xi \in \Xi}c_{\xi}F_{1}(\xi), \sum_{\xi \in \Xi}c_{\xi}F_{2}(\xi),\ldots \right)$$  
and
$$ \sum_{\xi \in \Xi} \sum_{\zeta\in \Xi}c_{\xi} \overline{c_{\zeta} }K(\xi, \zeta) =\sum_{\ell=1}^{\infty} \sum_{\ell'=1}^{\infty}y_{\ell}a_{\ell, \ell'}\overline{y_{\ell'}}< 0.$$
\end{proof}
Now we can prove the following  characterisation of positive definiteness which is more convenient since it only requires properties of $A$ as an operator on the Hilbert space $\ell^2$. The result could  be easily deduced from Theorem 2.1 of \citep{Dyn1997} but we give a self-contained proof to be able to use the  methodology again. 
\begin{thm}
Let $K: \S \times \S \rightarrow \C$ be a continuous kernel with $K(\xi,\zeta)=\overline{K(\zeta,\xi)}$. Then $K$ is positive definite if and only if $A$ is a positive operator on $\ell^2$, where $A$ is as in \eqref{eq:DefOp}. 
\end{thm}
\begin{proof}
We prove the first direction by contradiction.
We assume there is a $y\in \ell^2$ with $\sum_{\ell=1}^{\infty} \sum_{\ell'=1}^{\infty}y_{\ell}a_{\ell, \ell'}\overline{y_{\ell'}}< 0$ and deduce that $K$ is not positive definite. 

Assume $y\in \ell^2$ with $\langle A(y),y\rangle <0$, then we know  the function $f $ with 
$$f(\xi):=\sum_{\ell=1}^{\infty}y_{\ell}F_{\ell}(\xi)\in L^2(\S)$$ and 
$$ \langle A(y),y\rangle =\int_{\S} \int_{\S} \overline{f}(\zeta)K(\xi,\zeta)f(\xi)\,d\xi\,d\zeta= \langle \bar{f}(\xi)f(\zeta),K(\xi,\zeta) \rangle_{L^2(\S\times \S)}<0.$$

Since every function in $L^2(\S)$ can be approximated arbitrarily well using functions in $C^{\infty}(\S)$ there exists a sequence of  functions $\tilde{f}_n\in C^{\infty}(\S)$ with 
$$\underset{n\rightarrow \infty}{\lim} \Vert f-\tilde{f}_n\Vert_{L^2(\S)}=0.$$
It follows that also $\underset{n\rightarrow \infty}{\lim} \Vert \bar{f}-\bar{\tilde{f}}_n\Vert_{L^2(\S)}=0$ and 
$$\underset{n\rightarrow \infty}{\lim} \Vert f(\xi)\bar{f}(\zeta)-\tilde{f}_n(\xi)\bar{\tilde{f}}_n(\zeta)\Vert_{L^2(\S\times \S)}=0.$$
Using Cauchy-Schwarz inequality on the space $L^2(\S\times \S)$ we deduce that
$$ \underset{n\rightarrow \infty}{\lim} \langle \bar{\tilde{f}}_n(\xi)\tilde{f}_n(\zeta),K(\xi,\zeta) \rangle_{L^2(\S\times \S)}= \langle \bar{f}(\xi)\tilde{f}(\zeta),K(\xi,\zeta) \rangle_{L^2(\S\times \S)}<0.$$
 Therefore there  exists a value $n\in\N$ for which  
 $$\int_{\S} \int_{\S} \bar{\tilde{f}}_n(\zeta)K(\xi,\zeta)\tilde{f}_n(\xi)\,d\xi\, d\zeta<0.$$
 From now we assume $n$ to be a fixed integer for which the above inequality holds.

The next step is the discretisation of  the above integral of a product of continuous functions on $\S\times \S$. We will use a tensor product of the Gaussian quadrature rule for the sphere but any quadrature rule that converges for continuous functions when the number of discretisation points goes to infinity would be applicable.   

We construct the  sequence $y_m$ using the product Gaussian quadrature rule,
$$I_m=\frac{\pi}{m} \sum_{j=1}^{2m} \sum_{i=1}^m w_i f(\xi_{i,j}).$$
Here $\xi_{i,j}$ given in polar coordinates is $(\pi i/m, \theta_{i,m})$, where $\cos(\theta_{i,m})$ and $w_i$ are the Gaussian Legendre quadrature points and weights on $[-1,1]$. The sum from the quadrature rule converges to $\int_{\S}f(\xi)\,d\xi$ for $m\rightarrow \infty$ and all continuous $f$ on $\S$.
As tensor product extension we  know that 
$$\underset{m\rightarrow \infty}{\lim}\frac{\pi^2}{m^2} \sum_{j,j'=1}^{2m} \sum_{i,i'=1}^m w_i w_{i'} f(x_{i,j},x_{i',j'})=   \int_{\S\times \S}f(\xi,\zeta)\, d\xi \, d\zeta, \qquad \forall f\in C(\S\times\S).$$

If we choose 
$$\Xi_m=\left\lbrace \xi_{i,j} \vert\ j=1,\ldots ,2m,\  i =1,\ldots,m\right\rbrace \text{ and }c_{\xi_{i,j}}=w_i \tilde{f}_n(\xi_{i,j}) $$
then $y_m=   \sum_{\xi \in \Xi_m} c_{\xi} \left( F_1 (\xi), F_2(\xi), \ldots\right)$ converges elementwise to the sequence of expansion coefficients of $\tilde{f}_n$. Now 
\begin{align*}
 \underset{m\rightarrow \infty}{\lim}\langle A(y_m),y_m \rangle &= \underset{m\rightarrow \infty}{\lim} \sum_{\xi,\zeta\in \Xi_m} c_{\xi} \overline{c_{\zeta}}K(\xi,\zeta)\\
 &=\int_{\S} \int_{\S} \overline{\tilde{f}_n}(\zeta)K(\xi,\zeta)\tilde{f}_n(\xi)\, d\xi \, d\zeta<0.
\end{align*}
This completes the proof of the first direction since it is a contradiction to $K$ being positive definite.

It remains to show that if $A$ is positive on $\ell^2$ then $K$ is positive definite. We again prove the claim using contradiction, by assuming there is a set of points $\Xi$ and coefficients $c\in \C^{\vert \Xi \vert}$ for which 
$$\sum_{\xi,\zeta \in \Xi} c_{\xi}\overline{c_{\zeta}}K(\xi,\zeta)<0$$ 
and prove that in this case $A$ is not a positive operator on $\ell^2$. 

We note that according to Lemma 1 the assumption is equivalent to $\langle A(y),y\rangle<0$  with $y=\sum_{\xi \in \Xi} c_{\xi} \left( F_1 (\xi), F_2(\xi), \ldots\right)$.

We now construct a sequence of sequences  $y^n \in \ell^2$ for which $\underset{n\rightarrow \infty}{\lim} \langle A(y^n),y^n\rangle= \langle A(y),y\rangle$.  The tool we use is the heat kernel.

We start with defining the sequence $y^n=\left( e^{-\lambda_{\ell}\frac{1}{n}}y_{\ell}\right)_{\ell=1}^{\infty}$, which converges elementwise to $y$ for $n\rightarrow \infty$. Now we use the addition theorem for spherical harmonics, which states that for any orthonormal basis of the eigenfunctions corresponding to the eigenvalue  $k(k+1)$:
$$\Big\vert \sum_{\ell\text{ with }\lambda_{\ell}=k(k+1)}F_{\ell}(\xi)\overline{F_{\ell}}(\zeta)\Big\vert \leq \frac{2k+1}{4\pi}, \qquad \forall \xi,\zeta \in \Xi.$$
From this we deduce that the $\ell^2$-norm  of $y^n_{\xi}=\left( e^{-\lambda_{\ell}\frac{1}{n}}F_{\ell}{(\xi)}\right)_{\ell=1}^{\infty}$ for arbitrary $\xi \in \S$ can be estimated by 
\begin{equation}\label{eq:estNorm} \Vert y^n_{\xi}\Vert_2^2\leq \sum_{k=0}^{\infty} e^{-2k(k+1)\frac{1}{n}}\left(\frac{2k+1}{4\pi}\right)^2<\infty.
\end{equation}
Accordingly $y^n$ is also in $\ell^2$ as a linear combination $\sum_{\xi\in \Xi}c_{\xi}y_{\xi}^n$.
Now 
\begin{align*}\underset{n\rightarrow \infty}{\lim} \langle A(y^n),y^n\rangle=&\underset{n\rightarrow \infty}{\lim}\sum_{\ell,\ell'=1}^{\infty}y^n_{\ell}a_{\ell,\ell'}\overline{y^n_{\ell'}}\\
&= \underset{n\rightarrow \infty}{\lim}\sum_{\ell,\ell'=1}^{\infty}y_{\ell} \underset{:=a^n_{\ell,\ell'}}{\underbrace{e^{-\lambda_{\ell}\frac{1}{n}}a_{\ell,\ell'} e^{-\lambda_{\ell'}\frac{1}{n}}}}\overline{y_{\ell'}}.
\end{align*}
We define $K_n(\xi,\zeta):=\sum_{\ell,\ell'=1}^{\infty} a^n_{\ell,\ell'}F_{\ell}(\xi)\overline{F_{\ell'}}(\zeta)$ and use the arguments  used in \eqref{eq:estNorm} to show that $K_n\in W^2(\S\times \S) \subset C(\S \times \S)$.  It follows that
\begin{align*} \underset{n\rightarrow \infty}{\lim} \langle A(y^n),y^n\rangle=\underset{n\rightarrow \infty}{\lim} \sum_{\xi ,\zeta\in \Xi}c_{\xi}\overline{c_{\zeta}}K_n(\xi,\zeta).
\end{align*}
It remains to prove that $K_n$ converges pointwise to $K$. This is true since
$$\underset{n\rightarrow \infty}{\lim} \Vert K_n-K \Vert^2_{L^2(\S\times\S)}= \underset{n\rightarrow \infty}{\lim} \sum_{\ell,\ell'=1}^{\infty} \left( 1-e^{-\frac{1}{n}(\lambda_{\ell}+\lambda_{\ell'})}\right)^2a_{\ell,\ell'}^2=0.$$
The convergence in $L^2$- norm implies pointwise convergence since both kernels are continuous and we finally have 
\begin{align*} \underset{n\rightarrow \infty}{\lim} \langle A(y^n),y^n\rangle= \underset{n\rightarrow \infty}{\lim}\sum_{\xi, \zeta\in \Xi}c_{\xi}\overline{c_{\zeta}}K_n(\xi,\zeta)=\sum_{\xi, \zeta\in \Xi}c_{\xi}\overline{c_{\zeta}}K(\xi,\zeta)<0.
\end{align*}


\end{proof}
To solve the interpolation problem positive definiteness is not sufficient and we have to investigate strict positive definiteness. Analogue to Lemma 1 we can easily deduce:
\begin{thm}
Let $K: \S \times \S \rightarrow \C$ be a continuous kernel with $K(\xi,\zeta)=\overline{K(\zeta,\xi)}$. Then $K$ is strictly positive definite if and only if $\langle A(y),y\rangle >0$ for all  $y \in  Y\setminus \lbrace 0 \rbrace $, where $A$ is as in \eqref{eq:DefOp}. 
\end{thm}
\begin{proof}
If we assume that $K$ is strictly positive definite it will also be positive definite and thus $\langle A(y),y\rangle \geq0$ for all $y \in Y$ according to Lemma 1. Also for all elements $y\in Y \setminus \lbrace 0\rbrace$ there exists  a set $\Xi$ of distinct points and coefficients $c_{\xi}$ not all zero s.t. $y =\left(\sum_{\xi \in \Xi}c_{\xi}F_{1}(\xi), \sum_{\xi \in \Xi}c_{\xi}F_{2}(\xi),\ldots \right)$.
Then 
\begin{align*}
\sum_{\xi \in \Xi} \sum_{\zeta \in \Xi}c_{\xi} \overline{c_{\zeta}} K(\xi, \zeta)=\sum_{\ell,\ell'=1}^{\infty}y_{\ell} a_{\ell, \ell'}\overline{y_{\ell'}}=\langle A(y),y\rangle >0,
\end{align*}
as long as $y\neq0$.

For the second direction we assume $\langle A( y) ,y\rangle > 0$ for all  $y \in  Y\setminus \lbrace 0 \rbrace$. Then $K$ is positive definite as a result of Lemma 1. Additionally we can show that for a finite non empty set of distinct points $\Xi$
$$y=\left(\sum_{\xi \in \Xi}c_{\xi}F_{1}(\xi), \sum_{\xi \in \Xi}c_{\xi}F_{2}(\xi),\ldots \right)= 0$$   if and only if  $c=0\in \C^{\vert \Xi \vert}$, because the $F_{\ell}$ form a basis of $L^2(\S)$ and 
$$\sum_{\xi \in \Xi}c_{\xi}F_{\ell}(\xi)=0,\qquad  \forall \ell \in \N,$$
would implicate   $\sum_{\xi \in \Xi}c_{\xi}f(\xi)=0$ for all $f\in L_2(\S)$, but the evaluation functionals on $L^2(\Omega)$ are linearly independent.
Now 
$$\sum_{\xi \in \Xi} \sum_{\zeta \in \Xi}c_{\xi} \overline{c_{\zeta}} K(\xi, \zeta)=\langle A(y),y \rangle>0,$$
 as long as $c\neq0\in \C^{\vert \Xi \vert}$.
\end{proof}
The analogue of Theorem 2 does not hold for strict positive definiteness as follows from the characterisation of strictly positive definite isotropic kernels by Chen, Menegatto and Sun \citep{Chen2003}. 

\section{Kernels with special coefficient structure}

For the rest of the paper, we change to the double index notation of the spherical harmonics, which allows us to employ more of the specific properties of the eigenfunctions but comes at the cost of a more complicated notation. Our kernels which are still Hermitian and continuous are now represented as:
 \begin{equation}\label{eq:GenKer} K(\xi, \zeta)=\sum_{j,j'=0}^{\infty}\sum_{k=-j}^{j} \sum_{k'=-j'}^{j'} \alpha_{j,j',k,k'}Y_j^k(\xi)\overline{Y_{j'}^{k'}(\zeta)},\qquad \forall \xi,\zeta \in \S ,
 \end{equation}
where 
$$Y_j^k(\theta,\varphi)= \frac{1}{\sqrt{2\pi}} \sqrt{\frac{2j+1}{2} \frac{(j-k)!}{(j+k)!}} P_{j,k}(\cos(\theta))e^{ik \varphi}$$
 and $P_{j,k}$ are the associated Legendre polynomials.
Imposing certain conditions on the structure of the coefficients allows us to focus on kernels with specific properties. The most often studied case is assuming that the kernel is isotropic \citep{Schoenberg1938}, which means it only depends on the distance of its two arguments and not on their position on the sphere. 
The coefficients of an isotropic positive definite kernel given in the form  \eqref{eq:GenKer} satisfy
$$\alpha_{j,j',k,k'}=c_j \delta_{j,'j} \delta_{k,k'}\geq 0$$ 
as stated in \citep{Schoenberg1938}. A characterisation of the strictly positive definite kernels of this form was presented by Chen et al.\ in \citep{Chen2003}. 

Recently a sufficient condition  for axially symmetric kernels  to be strictly positive definite was presented in \citep{Bissiri2020}. The axially symmetric kernels do depend on the difference in longitude of the two inputs  $\xi,\zeta$ and their individual values of latitude.  The coefficients of these kernels satisfy $$\alpha_{j,j',k,k'}=c_k(j,j')\delta_{k,k'}.$$

In this publication, we add necessary conditions and sufficient conditions for a new class of kernels, which includes the results stated for isotropic kernels as a special case.
\begin{lemma}
A  kernel $K\in C^2(\S\times\S)$ satisfies  
$$\triangle_{\S,1}K(\xi,\zeta)=\triangle_{\S,2}K(\xi,\zeta),\qquad \forall \xi,\zeta \in \S,$$ 
where $\triangle_{\S,1}$/ $\triangle_{\S,2}K(\xi,\zeta)$ means that the Laplace-Beltrami operator is applied to the function with respect to the first or second argument, if and only if 
$$\alpha_{j,j',k,k'}=d_j(k,k')\delta_{j,j'}.$$
\end{lemma}
\begin{proof}
The spherical Laplacian applied to the kernel yields:
\begin{align*}\triangle_{\S,1}K(\xi,\zeta)&=\sum_{j,j'=0}^{\infty}\lambda_j \sum_{k=-j}^j \sum_{k'=-j'}^{j'}a_{j,j',k,k'} Y_{j}^k(\xi)\overline{Y_{j'}^{k'}(\zeta)},\\
\triangle_{\S,2}K(\xi,\zeta)&=\sum_{j,j'=0}^{\infty}\lambda_{j'} \sum_{k=-j}^j \sum_{k'=-j'}^{j'}a_{j,j',k,k'} Y_{j}^k(\xi)\overline{Y_{j'}^{k'}(\zeta)}.
\end{align*}
Since both $\overline{Y_{j'}^{k'}}$ as well as $Y_j^k$ form a set of linear independent functions and we assume equality of the two equations above  for all $\xi,\zeta\in\S$, all the expansion coefficients in the above equation need to be equal. Additionally except $\lambda_0$ all $\lambda_j $ are positive and distinct, therefore  
$$\lambda_j a_{j,j',k,k'}=\lambda_{j'}a_{j,j',k,k'},\qquad \forall j,j' \in \N,$$
is achieved if and only if  
$$\alpha_{j,j',k,k'}=d_j(k,k')\delta_{j,j'}.$$
\end{proof}
Kernels of this form are also invariant under parity, meaning $K(\xi,\zeta)=K(-\xi,-\zeta)$ and the special structure allows to determine easily if an interpolant derived using such a kernel is included in certain  Sobolev spaces, this is the case even if we do not impose the conditions of twice differentiability and real range.
\subsection{(Strict) positive definiteness of kernels with eigenvalue block structure}
We now assume that $K$ is continuous Hermitian and has the form  
\begin{equation}\label{eq:Kernform2} K(\xi, \zeta)=\sum_{j=0}^{\infty}\sum_{k=-j}^{j} \sum_{k'=-j}^j d_j(k,k') Y_j^k(\xi)\overline{Y_{j'}^{k'}(\zeta)},\qquad \forall \xi,\zeta \in \S .
 \end{equation}
 We now define the linear operators 
 \begin{align}\label{eq:Dj}
 D_j&: \C^{2j+1} \rightarrow \C^{2j+1},\qquad x\mapsto \left(\sum_{k=1}^{2j+1}(d_j(k-j-1,k'-j-1))x_k\right)_{k'=1}^{2j+1},
 \end{align}
 on $\C^{2j+1}$.
\begin{thm}A continuous and Hermitian kernel $K$ of the form \eqref{eq:Kernform2} is positive definite if and only if the operator  $D_j$ is a positive operator, for all $j\in \N$. 
\end{thm}
\begin{proof}
To employ the results of Section 2  we identify  the double indexed series with the elements $F_{\ell}$ by setting: 
$$F_1(\xi)=Y_ 0^0(\xi), \qquad F_{j^2+j+1+k}(\xi)=Y_j^k(\xi)$$
and $$a_{j^2+j+1+k,j^2+j+1+k'}=d_j(k,k')\delta_{j,j'}.$$
The operator $A$ from \eqref{eq:DefOp} thereby has the form  
\begin{align*} \langle A(y),y \rangle &=\sum_{j=0}^{\infty}\left( \sum_{k=-j}^j \sum_{k'=-j}^{j} y_{j^2+j+1+k}d_j(k,k') \overline{y_{j^2+j+1+k'}}\right) \\
&=\sum_{j=0}^{\infty}\langle D_j(y^j),y^j\rangle_{\C^{2j+1}}, \qquad y^j=\left(y_{j^2+1},\ldots,y_{j^2+2j+1}\right) \in \C^{2j+1}.
\end{align*}
From Theorem 1 we deduce that $K$ is positive definite if and only if $A$ is positive on $\ell^2$. This is the case if and only if the operators $D_j$ are positive on $\C^{2j+1}$.
\end{proof}
We note that since the $D_j$ are linear operators on finite dimensional spaces, they can be represented as matrices $D_j=\left(d_j(k,k')\right)_{k,k'=-j}^j$. We will now use the notation $D_j$ for the matrix as well as for the operator. $D_j$ is a positive operator if and only if the matrix $D_j$ is positive semi-definite and $D_j$ is not the zero operator if and only if  $D_j$ is not the zero matrix.
\begin{lemma} If a continuous and Hermitian kernel $K$ of the form \eqref{eq:Kernform2}
 is strictly positive definite, the operators $D_j$ are all positive and $D_j$ is not the zero operator for infinitely many even and infinitely many odd values of $j\in \N$.
\end{lemma}
\begin{proof}
We assume $K$ is continuous, Hermitian and strictly positive definite. Since strict positive definiteness implies positive definiteness the operators $D_j$ are all positive  according to Theorem 3.  
The rest of the proof is divided into four cases. Let us denote the set of indices for which $D_j$ is nonzero with $j\in J$.
For all four we assume that $K$ is strictly positive definite of the form  \eqref{eq:Kernform2} and prove that assuming  either
\begin{enumerate}
\item $J \subset 2\N$ or
\item $J\subset 2\N+1$ or
\item $1\leq \vert J\cap 2\N\vert< \infty$ or
\item $1\leq \vert J\cap (2\N+1)\vert< \infty$,
\end{enumerate}
leads to a contradiction.

1. Let us now assume $D_j$ is nonzero\textbf{ only for even} values of $j$. We can represent the quadratic form as 
\begin{align*} \sum_{\xi,\zeta \in \Xi}c_{\xi} \overline{c_{\zeta}} K(\xi, \zeta)&=\sum_{j=0}^{\infty}\sum_{k=-j}^{j} \sum_{k'=-j}^j d_j(k,k')\sum_{\xi \in \Xi}c_{\xi} Y_j^k(\xi) \overline{\sum_{\zeta \in \Xi}c_{\zeta}Y_{j'}^{k'}(\zeta)} \\
& =\sum_{j=0}^{\infty} y^j_{\Xi} D_j \overline{y^j_{\Xi}},\quad \text{ with } y^j_{\Xi}=\left(\sum_{\xi \in \Xi}c_{\xi} Y_j^k(\xi)\right)_{k=-j}^j \in\C^{2j+1}.
\end{align*}
Choosing a non empty set of data sites $\Xi'$ which satisfies 
$$\xi \in \Xi'\ \Rightarrow \ -\xi \in \Xi',$$
where $-\xi$ is the antipodal point of $\xi$, and setting $c_{\xi}=-c_{-\xi}$ we find that $y_{\Xi'}^j=0 \in \C^{2j+1}$ for all even $j$ since $Y_j^k(-\xi)=(-1)^jY_j^k(\xi)=Y_j^k(\xi)$. 
This implies $\sum_{\xi,\zeta \in \Xi'}c_{\xi} \overline{c_{\zeta}} K(\xi, \zeta)=0$ and therefore $K$ would not be strictly positive definite.

2. The same argument applies when we assume $D_j$ is nonzero \textbf{ only for odd} values of $j$. The same set $\Xi'$ can be chosen but  now $c_{\xi}=c_{-\xi}$ yields the contradiction to the assumption of strict positive definiteness.

3. Now we assume that $D_j$ is nonzero for any number of odd values of $j$ and\textbf{ only finitely many even} values of $j$. Set $\hat{j}$ to the maximal even index for which $D_j$ is nonzero.
We aim to construct a set $\Xi$ with elements only in the lower hemisphere of $\S$ and $c\in \C^{\vert \Xi \vert}\neq 0$, s.t. 
$\sum_{\xi \in \Xi}c_{\xi} Y_j^k(\xi) =0 \in \C^{2j+1}$ for all even $j \leq \hat{j}$. 
These are $$\sum_{m=1}^{\hat{j}/2}(2(2m)+1)=\frac12 \hat{j}^2 +\frac32 \hat{j}$$
 linear equations, further 
there exists no point on $\S$ with $Y_j^k(\xi)=0$ for all even $j$ with $j\leq \hat{j}$ since $Y_0^0(\xi)=1$  for all $x \in \S$. We can therefore choose any set of distinct points in the lower hemisphere with more than $\frac12 \hat{j}^2 +\frac32 \hat{j}$ elements to find a non trivial solution.  Defining the set $\Xi'=\Xi \cup( -\Xi)$ and setting $c_{\xi}=-c_{-\xi}$ shows that 
 $$\sum_{\xi,\zeta \in \Xi'}c_{\xi} \overline{c_{\zeta}} K(\xi, \zeta)=0$$
 for a non trivial vector $c$ and therefore $K$ is not strictly positive definite.
 
4. The same arguments can be used to show that $D_j$ needs to be\textbf{ nonzero for infinitely many odd} values of $j$ with the extra argument that there exists no $Y_j^k(\xi)=0$ for all odd $j$ with $j\leq \hat{j}$ because this is not even possible for $\hat{j}=1$ with:
\begin{align*} Y_1^0 (\theta,\phi)=\frac12 \sqrt{\frac{3}{\pi}} \cos(\theta)&, \ Y_1^{-1}(\theta,\phi)=\frac12 \sqrt{\frac{3}{2\pi}} \sin(\theta) e^{-i\phi},\\  \ Y_1^{1}(\theta,\phi)&=\frac{-1}{2} \sqrt{\frac{3}{2\pi}} \sin(\theta) e^{i \phi}.
\end{align*}
\end{proof}

With the aim of finding sufficient conditions that are easy to evaluate we define the set $\mathcal{F}$ as the set of all indices $j \in  \N$ for which $D_j$ is a strictly positive operator on $\C^{2j+1}$.
\begin{lemma}
A continuous and Hermitian kernel $K$ of the form \eqref{eq:Kernform2} 
 is strictly positive definite if the operators \eqref{eq:Dj} are positive for all $j\in \N$ and if
$$\sum_{\xi \in \Xi}c_{\xi} Y_j^k(\xi)=0,\  \forall \ k=-j, \ldots,j,\  j\in \mathcal{F} \quad  \Rightarrow\quad c_{\xi }=0, \forall \xi \in \Xi.$$
\end{lemma}
\begin{proof}
We prove the result by showing that $K$ not being strictly positive definite is  contradicting the last statement. If $K$ is not strictly positive definite there exists a nonempty set of distinct point $\Xi$ and coefficients $c_{\xi}$ not all zero with  
$$ \sum_{\xi,\zeta \in \Xi}c_{\xi} \overline{c_{\zeta}} K(\xi, \zeta)=0. $$
This is equivalent according to the computation in the proof of Lemma 3 to 
\begin{equation}\label{eq:SumDj}  \sum_{j=0}^{\infty} (y^j_{\Xi})^T D_j \overline{y^j_{\Xi}}=0.
\end{equation}
From Lemma 3 we know that if $K$ is positive definite the $D_j$ are positive semi definite and the quadratic form \eqref{eq:SumDj} can only be zero if  $(y^j_{\Xi})^T D_j \overline{y^j_{\Xi}}=0$ for all $j\in \N$.  This is only possible if $y^j_{\Xi}$ is a linear combination of eigenvectors of $D_j$ corresponding to the eigenvalue $0$ (or $y^j=0$ if the matrix $D_j$ is a positive definite matrix). Therefore \eqref{eq:SumDj} holds only if  
$$\sum_{\xi\in \Xi} c_{\xi}Y_j^k(\xi) = 0, \qquad  \forall \ k=-j, \ldots,j,\  j\in \mathcal{F}$$
which contradicts the last statement of the theorem.
\end{proof}
The next proposition follows immediately from the above lemma together with Theorem 2.
\begin{prop}A continuous and Hermitian kernel $K$ of the form \eqref{eq:Kernform2} 
 is strictly positive definite if the operators $D_j$ are strictly positive definite for all $j\in \N$.
\end{prop}

With Lemma 4 we have shown that strict positive definiteness of these kernels can be proven using the same results which were used for zonal kernels by Chen et al.\ in \citep{Chen2003} and the alternative proof for more general manifolds stated by Barbosa and Menegatto in \citep{Barbosa2016}. We nevertheless include the rest of the proof for completeness and since the mentioned publications focused on approximation of real functions with zonal kernels the results were stated only for real $c$. 
\begin{lemma}
For a given set $\mathcal{F}$ the following two properties are equivalent:
\begin{enumerate}
\item $\sum_{\xi \in \Xi}c_{\xi} Y_j^k(\xi)=0,\  \forall \ k=-j, \ldots,j,\  j\in \mathcal{F}$ implies  $c_{\xi }=0,\ \forall \xi \in \Xi.$
\item $ \sum_{\xi\in \Xi }c_{\xi} P_{j}(\xi^T\zeta)=0,\ \forall j \in \mathcal{F},\ \zeta \in \S$ 
implies $c_{\xi }=0,\ \forall \xi \in \Xi$.
\end{enumerate}
\end{lemma}
\begin{proof}
Using the addition formula of the Legendre polynomials we find $$P_{j}(\xi^T\zeta)=\frac{4\pi}{2j+1}\sum_{k=-j}^j Y_{j}^k(\zeta)\overline{Y_{j}^k(\xi).
}$$
Therefore 
$$ \sum_{\xi \in \Xi}c_{\xi} P_{j}(\xi^T\zeta)=\sum_{k=-j}^j \left(\sum_{\xi \in \Xi}c_{\xi}Y_{j}^k(\xi)\right) \overline{Y_{j}^k(\zeta)}, \qquad \forall \zeta\in \S.$$
Since the spherical harmonics are linearly independent the last is  zero if and only if 
$$\sum_{\xi \in \Xi}c_{\xi}Y_{j}^k(\xi)=0$$ for all $k=-j,\ldots ,j$ and all $j\in \mathcal{F}$. 

\end{proof}

\begin{thm}
Let $K$ be a continuous positive definite kernel of the form \eqref{eq:Kernform2} and $\mathcal{F}$ the corresponding index set for which the $D_j$ are strictly positive operators. Then it is sufficient for $K$ to be strictly positive definite that $\mathcal{F}$ includes infinitely many even and infinitely many odd values of $j\in\N$.
\end{thm} 
\begin{proof}Combining Lemma 6 and Lemma 7 we know that for $K$ to be strictly positive definite  it is sufficient to prove that for arbitrary sets of distinct data sites $\Xi \subset \S$ the functions 
$P_j(\xi^T\zeta)$ satisfy Lemma 7 (2). Therefore we show that for any such set $\Xi$
$$\sum_{\xi \in \Xi}c_{\xi}P_j(\xi^T\zeta)=0, \qquad  \forall j \in \mathcal{F},\ \zeta\in \S,$$
implies $c_{\xi}=0$ for all  $\xi \in \Xi$.
We do this by choosing for each $\xi\in \Xi$ a corresponding $\zeta=\zeta_{\xi}\in \S$ and show that for this choice $c_{\xi}=0$ if the above holds.
Assume we have fixed $\xi \in \Xi$, we need to distinguish two cases.

\textbf{Case 1:} $ \xi^T\zeta\neq -1$ for all $\zeta \in \Xi.$

In this case we choose $\zeta_{\xi}=\xi$ and the system above takes the form  
$$c_{\xi} P_j(1) +\sum_{\zeta \in \Xi\setminus \lbrace \xi \rbrace}c_{\zeta}P_j(\zeta^T\xi)=0, \qquad  \forall j \in \mathcal{F}.$$
Note that $\xi^T\zeta \in (-1,1)$ and we set $\theta_{\zeta}=\arccos(\xi^T\zeta)\in (0,\pi)$. 
Since there are infinitely many even and odd indices in $\mathcal{F}$ we can choose a sequence from $\mathcal{F}$ with $j_n\in \mathcal{F}$ and $\underset{n \rightarrow \infty}{\lim} j_n=\infty$.
Introducing the limit in the above equation and using the asymptotic form of the Legendre polynomials for $n \rightarrow \infty$ (8.721.3, \citep{Gradshteyn2014}), 
\begin{align*} c_{\xi} + \underset{n \rightarrow \infty}{\lim} \sum_{\zeta \in \Xi \setminus \lbrace \xi \rbrace}c_{\zeta}\left( \frac{2}{\sqrt{2\pi j_n \sin(\theta_{\zeta})}}\cos\left(\left(j_n + \tfrac12\right)\theta_{\zeta} - \frac{\pi}{4}\right) + \mathcal{O}\left(j_n^{-3/2}\right) \right)=0 
\end{align*}
implies $c_{\xi}=0$ because the sum is finite.

\textbf{Case 2} There is one $\zeta \in \Xi $ with $\xi^T \zeta =-1$.

Then $\zeta$ is the antipodal point of $\xi$ and in Cartesian coordinates $\xi =-\zeta$.  Then the above equation becomes 
$$c_{\xi} P_j(1) +(-1)^j c_{-\xi} P_j(1)+\sum_{\zeta \in \Xi\setminus \lbrace \xi, -\xi \rbrace}c_{\zeta}P_j(\zeta^T\xi)=0, \qquad  \forall j \in \mathcal{F}.$$
The third part vanishes if we introduce the limit as in the previous case but for odd and even series of $j_n$ separately. The remainder yields for even $j$ and odd $j$ respectively  
$$ c_{\xi }+c_{-\xi}=0= c_{\xi}-c_{-\xi}$$
which implies $c_{\xi}=0=c_{-\xi}$.
\end{proof}

From this we can easily deduce a generalisation of the results of Chen et al.\ which were stated for zonal kernels for the case of kernels which are of diagonal form.
\begin{prop}
A continuous kernel of the form 
$$ K(\xi,\zeta)=\sum_{j=0}^{\infty} \sum_{k=-j}^j a_{j,k} Y_j^k(\xi) \overline{Y_j^k}(\zeta), \qquad a_{j,k} \in \R,$$ 
is strictly positive definite if $a_{j,k}\geq0$ always and $a_{j,k}>0$ for all $k=-j,\ldots,j$ for infinitely many even and infinitely many odd values of $j$.
\end{prop}

\subsection*{Acknowledgement}
The work of J. J\"ager was supported by the Justus Liebig University's postdoctoral fellowship Just'Us

\def\ln{\log}

\par\bigskip\bigskip\noindent
\bibliographystyle{abbrvnat}
\bibliography{C:/Users/Janin/JLUbox/Dokumente/LiteraturALL.bib}

\end{document}